\documentclass[11pt,reqno]{amsart}
\usepackage{graphicx}
\usepackage[caption = false]{subfig}
\usepackage{amsmath,amsfonts,amssymb,comment}
\usepackage{mathrsfs}
\usepackage{siunitx}
\usepackage{float}
\usepackage{enumerate}
\usepackage{enumitem}
\usepackage{mathtools}
\usepackage{amsthm}
\usepackage{varioref}
\usepackage[english]{babel}
\usepackage[utf8x]{inputenc}\usepackage[english]{babel}
\usepackage{soul}
\usepackage{enumitem}
\usepackage{xcolor}
\usepackage{hyperref}
\usepackage{cite}
\usepackage{relsize} 
\usepackage{titlesec}
\usepackage{wrapfig}
\usepackage[margin=1in]{geometry}
\titleformat{\section}[block]
  {\normalfont\Large\bfseries}{\thesection}{1em}{}
\titlespacing*{\section}
  {0pt}{10pt plus 4pt minus 2pt}{7pt plus 4pt minus 2pt}
\titleformat{\subsection}[block]
  {\normalfont\large\bfseries}{\thesubsection}{1em}{}
\titlespacing*{\subsection}
  {0pt}{8pt plus 3pt minus 2pt}{5pt plus 2pt minus 1pt}
\newtheorem{theorem}{Theorem}[section]

\newtheorem{lemma}[theorem]{Lemma}

\theoremstyle{definition}

\theoremstyle{remark}

\numberwithin{equation}{section}

\begin{document}
	\title[On Coefficient problems]{Coefficient problems of Starlike Functions Related to a Balloon-Shaped Domain }
	    \author[S. S. Kumar]{S. Sivaprasad Kumar}
	\address{Department of Applied Mathematics, Delhi Technological University, Bawana Road, Delhi-110042, INDIA}
	\email{spkumar@dce.ac.in}
        \author[A. Tripathi]{A. Tripathi}
	\address{Department of Applied Mathematics, Delhi Technological University, Bawana Road, Delhi-110042, INDIA}
	\email{ms.aryatripathi\_25phdam03@dtu.ac.in}
\begin{abstract}
Recent advances in image and signal processing have drawn on geometric function theory, particularly coefficient estimate problems. Motivated by their significance, we introduce a class of starlike functions related to a balloon-shaped domain  
\[
\mathcal{S}^*_{\mathcal{B}}= \left\{ f \in \mathcal{A} : \frac{z f'(z)}{f(z)} \prec \frac{1}{1-\log(1+z)} := B(z); \; z \in \mathbb{D} \right\},
\]  
where $B(z)$ maps the unit disk $\mathbb{D}$ onto a balloon-shaped domain. This work establishes bounds for the second order Hankel determinants and second order Toeplitz determinants involving the initial coefficients, the logarithmic coefficients and the logarithmic coefficients of the inverse function for $f \in \mathcal{S}^*_{\mathcal{B}}$.
\end{abstract}

\maketitle

\section{\hspace{5pt} Introduction}
\label{intro}
Consider the class $\mathcal{A}$ consisting of analytic functions $f$ defined in $\mathbb{D} := \{z \in \mathbb{C} : |z| < 1\}$, normalized by these conditions $f(0)=0$ and $f'(0)=1$. For $f\in \mathcal{A}$:
\begin{equation}
f(z) = z+a_2z^2+a_3z^3+\dots =z + \sum_{n=2}^{\infty} a_n z^n. \label{eq:1.1}
\end{equation}
Let $\mathcal{S}$ be a subclass of $\mathcal{A}$ consisting of analytic functions that are univalent in $\mathbb{D}$. The subclass $\mathcal{S}^{\ast}$ of $\mathcal{S}$ consisting of starlike functions and function $f\in \mathcal{S}^{\ast}$ is defined as:
\[
\mathcal{S}^{\ast} = \left\{ f \in \mathcal{A} : \Re \left(\dfrac{z f'(z)}{f(z)}\right) > 0; \quad z \in \mathbb{D} \right\}.
\]
For analytic functions $f$ and $g$, we say that $f$ is subordinate to $g$, written $f \prec g$ \cite{lowner1923}, if $f(z)=g(w(z))$ for some Schwarz function $w$ with $w(0)=0$ and $|w(z)|<1$. Ma and Minda \cite{MaMinda1994} introduced the class $\mathcal{S}^{\ast}(\varphi)$, defined by:
\[
\mathcal{S}^{\ast}(\varphi)=\left\{f\in\mathcal{A}:\frac{zf'(z)}{f(z)}\prec\varphi(z)\right\},
\]
where $\varphi$ is analytic, univalent, $\Re\varphi(z)>0$, and maps $\mathbb{D}$ onto a starlike domain symmetric about the real axis with $\varphi(0)=1$ and $\varphi'(0)>0$. Different choices of $\varphi$ yield well-known subclasses of $\mathcal{S}^{\ast}$, see
Table~\ref{tab:marks}.
\begin{table}[ht]
\centering
\begin{tabular}{|c|c|c|}
\hline
\textbf{Class} & \textbf{$\phi(z)$} & \textbf{References} \\ \hline
$\mathcal{S}^{\ast}[A,B]$ & $\dfrac{1+Az}{1+Bz}; \; -1\le B<A\le1$ & Janowski \cite{Janowski1970} \\ \hline
$\mathcal{S}^{\ast}_{\rho}$ & $1+\sinh^{-1}(z)$ & Arora \emph{et al.} \cite{AroraKumar2022} \\ \hline
$\mathcal{S}^{\ast}_e$ & $e^z$ & Mendiratta \emph{et al.}\cite{Mendiratta2015}\\ \hline
$\mathcal{SL}$ & $\sqrt{1+z}$ &  Stankiewicz \emph{et al.}\cite{Sokol1996} \\ \hline
$\mathcal{S}^{\ast}_{\mathfrak{B}}$ & $\sqrt{1+\tanh{z}}$ & Yadav \emph{et al.} \cite{KumarYadav2025}\\ \hline
$\mathcal{S}^{\ast}_q$ & $z+\sqrt{1+z^2}$ & Raina \emph{et al.} \cite{Raina2015} \label{tab:marks}  \\ \hline
\end{tabular}
\caption{Subclasses of starlike functions corresponding to various $\phi(z)$}
\end{table}

In this study, we investigate a class of starlike functions that are related to a balloon-shaped domain $B(\mathbb{D})$, illustrated in Figure.\ref{fig:phi}. We define this class as:
\[\mathcal{S}^{\ast}_{B}= \left\{ f \in \mathcal{A} : \frac{z f'(z)}{f(z)} \prec  B(z) \right\}.\] 

\begin{wrapfigure}{r}{0.4\textwidth}
\centering
\includegraphics[width=0.35\textwidth]{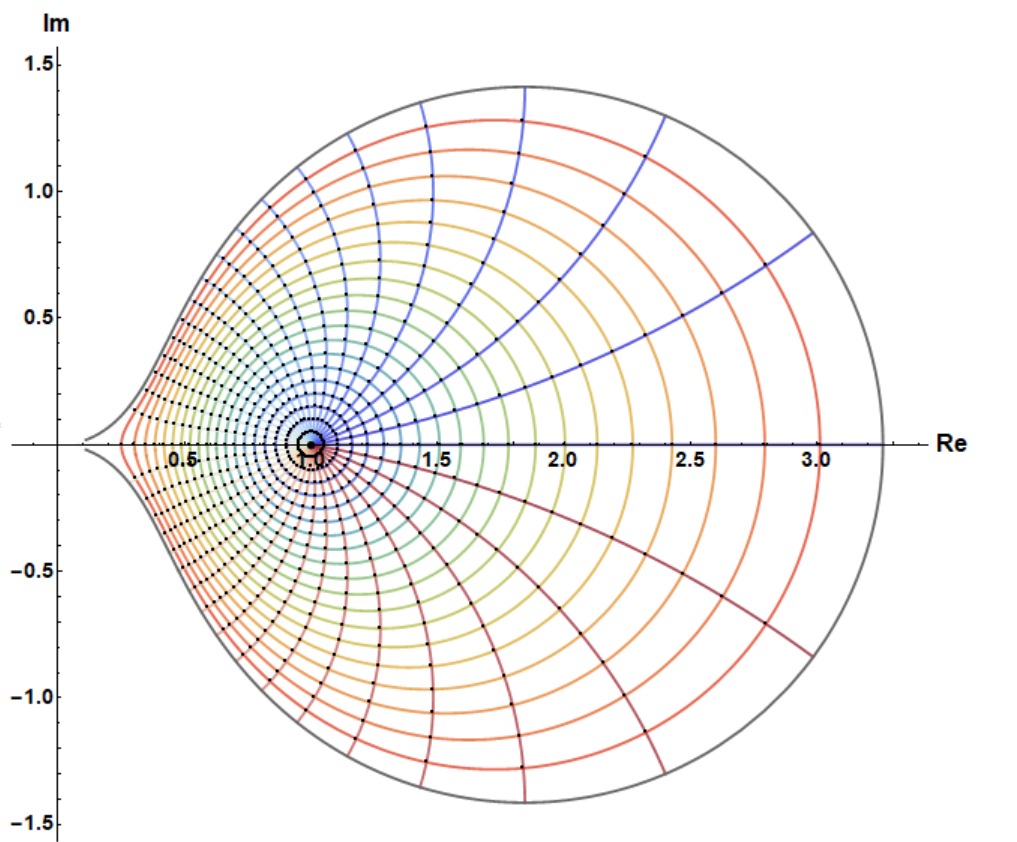}
\caption{\centering $B(\mathbb{D})$, a balloon-shaped domain, $B(z)=\dfrac{1}{1-\log(1+z)}$.}\label{fig:phi}.
\end{wrapfigure}

The domain $B(\mathbb{D})$ is characterized by $B(\mathbb{D})=\left\{w\in\mathbb{C}\setminus\{0\}: \big|\exp\!\big(1-\tfrac{1}{w}\big)-1\big|<1\right\}$, with boundary $\partial B(\mathbb{D})$ given by $|\exp(1-1/w)-1|=1$. Parametrizing $z=e^{i\theta}$ with $\theta\in(-\pi,\pi)$ gives $w(\theta)=[1-\log(2\cos(\theta/2))-i\theta/2]^{-1}$, revealing symmetry about the real axis, a rightmost tip at $w(0)=1/(1-\log2)\approx3.2589$, a convex main body extending leftward to approximately $(-0.181,\pm0.678)$, and a cusp at the origin as $\theta\to\pm\pi$ with $\arg(w)\to\mp\pi/2$ resembling a balloon with tied knot. Near the origin the boundary approximates $(x-1/2)^2+y^2=1/4$. The domain is starlike with respect to $w=1$ and univalent, see Figure~\ref{fig:phi}.\\

The functions in $\mathcal{S}^{\ast}_{B}$ can be represented through an integral formula as follows: 
\begin{equation}
    f(z)=z\exp \int_{0}^{z}\dfrac{\phi(t)-1}{t} dt. \label{extremalfunction}
\end{equation}
The functions $f_{1}(z)$, $f_{2}(z)$, and $f_{3}(z)$ serve as extremal functions for the family $\mathcal{S}^{\ast}_{B}$, obtained by $\phi(t)=B(t),\;\phi(t)=B(t^2)$ and $\phi(t)=B(it)$ in \eqref{extremalfunction}, respectively. These functions are given by:
\begin{eqnarray}
    f_{1}(z) = z \exp\left(\int_{0}^{z} \dfrac{\log(1+t)}{t(1-\log(1+t))} \, dt\right) &=&z +z^2+ \dfrac{3}{4} z^3 + \dfrac{19}{36}z^4+\dfrac{101}{288} z^5 + \ldots \;, \label{1.45}\\
    f_{2}(z) = z \exp\left(\int_{0}^{z} \dfrac{\log(1+t^2)}{t(1-\log(1+t^2))} \, dt\right) &=&z + \dfrac{1}{2} z^3 + \dfrac{1}{4} z^5 + \dfrac{5}{36}z^{7}\ldots \;, \label{1.466}\\
    f_{3}(z) =z \exp\left(\int_{0}^{z} \dfrac{\log(1+i t)}{t(1-\log(1+i t))} \, dt\right) &=& z +i z^2- \dfrac{3}{4} z^3 - \dfrac{19}{36} i z^4+\dfrac{101}{288}z^5 + \ldots \; \label{1.477}
\end{eqnarray}
For $f\in\mathcal{S}$, the logarithmic coefficients $\gamma_n(f)$ are defined as:
\begin{equation}
F_f(z) = \log\left( \dfrac{f(z)}{z}\right) = 2 \sum_{n=1}^{\infty} \gamma_n(f) z^n, \quad \text{where} \; z \in \mathbb{D} . \nonumber
\end{equation}
We denote $\gamma_n(f)$ by $\gamma_n$. For $f \in \mathcal{S}$, the logarithmic coefficients are given by:
\begin{eqnarray} 
    \gamma_1=\dfrac{1}{2} a_2, \quad \gamma_2 = \dfrac{1}{2} \left(a_3 - \dfrac{1}{2} a_2^2 \right), \quad \gamma_3 = \dfrac{1}{2} \left( a_4 - a_2 a_3 + \dfrac{1}{3} a_2^3 \right). \label{log coeff}
\end{eqnarray}
By the Koebe $1/4$-theorem, the inverse function of $f \in \mathcal{S}$ can be defined as $F_{f^{-1}} \in \mathcal{A}$ in a neighbourhood of the origin, given by:
\begin{equation}
F_{f^{-1}}(w):=\log\left(\dfrac{F(w)}{w}\right) = 2\sum_{n=1}^{\infty} \Gamma_n w^n, \quad |w| < \dfrac{1}{4}.\nonumber
\end{equation}
Here the logarithmic coefficients of the inverse function of $f \in \mathcal{S}$ are given as:
\begin{eqnarray}
\Gamma_1 = -\frac{1}{2}a_2,\quad \Gamma_2 = -\frac{1}{2}\left(a_3-\frac{3}{2}a_2^2\right),\quad \Gamma_3=-\frac{1}{2}\left(a_4-4a_2a_3+\frac{10}{3}a_{2}^{3}\right). \label{loginvcoef123}
\end{eqnarray}
The Bieberbach conjecture \cite{Goodman1983} motivated coefficient problems in univalent function theory. Pommerenke \cite{Pommerenke1967} introduced the $q$th Hankel determinant. For $f\in\mathcal{A}$, the Hankel determinant $\mathcal{H}_{q,n}(f)$, is defined as follows:
\begin{equation}
\mathcal{H}_{q,n}(f) = 
\begin{vmatrix}
a_n & a_{n+1} & \cdots & a_{n+q-1} \\
a_{n+1} & a_{n+2} & \cdots & a_{n+q} \\
\vdots & \vdots & \ddots & \vdots \\
a_{n+q-1} & a_{n+q} & \cdots & a_{n+2(q-1)}
\end{vmatrix},\quad q,n\in\mathbb{N}.\label{hqn}
\end{equation}
Sharp bounds for Hankel determinants in various subclasses of $\mathcal{S}$ have been extensively studied (see \cite{VermaKumar2025,KumarKamaljeet2021,Raza2025,KumarVerma2023, KumarVerma2025,Janteng2007}). From \eqref{hqn}, second order Hankel determinant
\begin{equation}
\mathcal{H}_{2,1}(f) = (a_3-a_2^2),\quad 
\mathcal{H}_{2,2}(f)= (a_2a_4-a_3^2).\label{h21 and h22}
\end{equation}
Ponnusamy \emph{et al.}~\cite{Ponnusamy2020} obtained sharp bounds for the logarithmic coefficients of univalent functions and their inverses. For $f \in \mathcal{A}$, the $q^{th}$ Hankel determinant $\mathcal{H}_{q,n}(F_f)$ is defined by these coefficients as follows:
\begin{equation}
    \mathcal{H}_{q,n}(F_{f}/2) = \begin{vmatrix}
\gamma_n & \gamma_{n+1} & \dots & \gamma_{n+q-1}\\
\gamma_{n+1} & \gamma_{n+2} & \dots & \gamma_{n+q} \\
\vdots & \vdots & \ddots & \vdots\\
\gamma_{n+q-1} & \gamma_{n+q} & \dots & \gamma_{n+2(q-1)}\\
\end{vmatrix}. \label{1.10}
\end{equation}
The idea of
studying Hankel matrices, whose entries are logarithmic coefficients of
analytic functions, was initiated by B. Kowalczyk and A. Lecko \cite{Kowalczykandlecko2022, Kowalczyk2022, Kowalczyk2023}, for further developments see \cite{MundaliaKumar2023,Kumar2025,AlluShaji2025,Mandal2025}. For $f \in \mathcal{A}$, the $q^{th}$ Hankel determinant $\mathcal{H}_{q,n}(F_f)$, defined by the logarithmic coefficients of $f^{-1}$, is given by:
\begin{equation}
\mathcal{H}_{q,n}(F_{f^{-1}}/2) = \begin{vmatrix}
\Gamma_n & \Gamma_{n+1} & \dots & \Gamma_{n+q-1}\\
\Gamma_{n+1} & \Gamma_{n+2} & \dots & \Gamma_{n+q} \\
\vdots & \vdots & \ddots & \vdots\\
\Gamma_{n+q-1} & \Gamma_{n+q} & \dots & \Gamma_{n+2(q-1)}\\
\end{vmatrix}. \label{1.13}
\end{equation}

The concept of studying Hankel matrices, whose entries are logarithmic coefficients of inverse analytic functions, was firstly studied in \cite{Eker2023, Lecko2024, Lecko2025}. For $f\in \mathcal{A}$, the Toeplitz determinant is given by:
\begin{equation}
\mathcal{T}_{q,n}(f) = 
\begin{vmatrix}
a_n & a_{n+1} & \cdots & a_{n+q-1} \\
a_{n+1} & a_n & \cdots & a_{n+q-2} \\
\vdots & \vdots & \ddots & \vdots \\
a_{n+q-1} & a_{n+q-2} & \cdots & a_n
\end{vmatrix},\quad q,n \in \mathbb{N}.\label{toep}
\end{equation}

Bounds for the Toeplitz determinant in the class $\mathcal{S}^{\ast}$ and $\mathcal{C}$ were studied by Ali \emph{et al.} \cite{Ali2018}, with further extensions to subclasses of $\mathcal{S}$ in \cite{Ahuja2021, GiriKumar2023,GiriKumar2024, Lecko2020, Obradovic2021, Toeplitz1907}. Setting $a_1=1$ and $q=2$ in \eqref{toep}, which yields the second-order Toeplitz determinant for $n=1,2$.
\begin{equation}
\mathcal{T}_{2,1}(f)=(1-a_2^2),\quad
\mathcal{T}_{2,2}(f)= (a_2^2-a_3^2), \quad
\mathcal{T}_{2,3}(f)=(a_3^2-a_4^2).  \label{t22}
\end{equation}
For $f \in \mathcal{S}$, the Toeplitz determinant corresponding to the logarithmic coefficients as follows:
\begin{equation}
   \mathcal{T}_{q,n}(F_{f}/2) = \begin{vmatrix}
\gamma_n & \gamma_{n+1} & \dots & \gamma_{n+q-1}\\
\gamma_{n+1} & \gamma_{n} & \dots & \gamma_{n+q-2} \\
\vdots & \vdots & \ddots & \vdots\\
\gamma_{n+q-1} & \gamma_{n+q-2} & \dots & \gamma_{n}\\
\end{vmatrix}. \label{1.12}
\end{equation} 
 The Toeplitz determinant $\mathcal{T}_{q,n}(F_{f^{-1}}/2)$, where the entries are the logarithmic coefficients of inverse of $f\in \mathcal{S}$ and are given as:
\begin{equation}
\mathcal{T}_{q,n}(F_{f^{-1}}/2) = \begin{vmatrix}
\Gamma_n & \Gamma_{n+1} & \dots & \Gamma_{n+q-1}\\
\Gamma_{n+1} & \Gamma_{n} & \dots & \Gamma_{n+q-2} \\
\vdots & \vdots & \ddots & \vdots\\
\Gamma_{n+q-1} & \Gamma_{n+q-2} & \dots & \Gamma_{n}\\
\end{vmatrix}. \label{1.14}
\end{equation}
Recent research in geometric function theory examines Hankel and Toeplitz determinants for coefficient bounds of analytic and starlike functions \cite{Kanwal2025}. Although applications remain limited, they show potential in image processing, signal analysis, and mathematical physics, including contrast enhancement \cite{Chen2019}, texture analysis \cite{Priya2023}, and fractional-order heat equations \cite{Ibrahim2022}.
This work introduces a balloon-shaped starlike class and derives sharp bounds for coefficients, the Fekete-Szeg\"o functional, and second-order Hankel and Toeplitz determinants.

\section{\hspace{5pt} Preliminary results}\label{3}
Let $\mathcal{P}$ be the class consisting of functions with positive real part in $\mathbb{D}$, and this subclass is known as Carath\'eodory class. The Taylor series expansion of $f \in \mathcal{P}$ is defined as:
\begin{equation}
    p(z) =1+p_1z+p_2z^2+\dots= 1 + \sum_{n=1}^{\infty}p_n z^n.\label{pp}
\end{equation}
The Carath\'eodory class $\mathcal{P}$ and its associated coefficient bounds play a crucial role in establishing the sharp bounds for the Hankel determinant. This section provides key lemmas that form the foundation for the main results presented in this section.
\begin{lemma}\label{lem:1}
\cite{cho_kowalczyk_lecko_smiarowska2020}: If $p \in \mathcal{P}$ is of the form \eqref{pp}, then 
\begin{eqnarray}
p_{1}&=&2\zeta_{1}, \label{l1}\\
p_{2}&=&2\zeta_{1}^{2}+2(1-\zeta_{1}^{2})\zeta_{2}, \label{l2}\\
p_{3}&=&2\zeta_{1}^{3}+4(1-\zeta_{1}^{2})\zeta_{1}\zeta_{2}-2(1-\zeta_{1}^{2})\zeta_{1}\zeta_{2}^{2}+2(1-\zeta_{1}^{2})(1-|\zeta_{2}|^{2})\zeta_{3},  \label{l3}
\end{eqnarray}
for some $\zeta_{1},\;\zeta_{2},\;\zeta_{3}$ $\in \mathbb{\overline{D}}$.\\ 
For $\zeta_{1}\in \mathbb{T}:=\{z \in \mathbb{C}\;;\;|z|=1\}$, there is a unique function $p \in \mathcal{P}$ with $p_1$ as in \eqref{l1}, namely,
\begin{equation}
    p(z)=\dfrac{1+\zeta_{1}z}{1-\zeta_{1}z}, \quad z\in \mathbb{D} \label{firstp(z)}
\end{equation}
For $\zeta_1 \in \mathbb{D}$ and $\zeta_2 \in \mathbb{T}$, there is a uniquee function $p\in \mathcal{P}$ with $p_1$ and $p_2$ as in \eqref{l1} and \eqref{l2}, namely,
\begin{equation}
    p(z)=\dfrac{1+(\overline{\zeta_1}\;\zeta_2+\zeta_1)z+\zeta_2z^2}{1+(\overline{\zeta_1}\;\zeta_2-\zeta_1)z-\zeta_2z^2}, \quad z \in \mathbb{D}. \label{p(z)}
\end{equation}

\end{lemma}

\begin{lemma}\label{lem:2}
\cite{Choi2007}:
If $A$, $B$, $C \; \in \mathbb{R}$, let us consider
$$Y(A, B, C):= \max \{ |A + Bz + Cz^{2} | + 1 - |z|^2, \quad z \in \mathbb{\overline{D}} \}$$

Case 1: If $AC \ge 0$, then
    \[Y(A, B, C)= \displaystyle\left\{\begin{array}{ll}
        |A| + |B| + |C|,& |B| \geq 2(1  - |C|), \\
        \\
        1 + |A| + \dfrac{B^2}{4(1 - |C|)},& |B| < 2(1 - |C|).
    \end{array} \right.\]
Case 2: If $AC < 0$, then
\[Y(A, B, C)=  \displaystyle\left\{\begin{array}{ll}
        1 - |A| + \dfrac{B^2}{4(1 - |C|)}, & -4AC(C^{-2} - 1) \le B^2 \wedge |B| < 2(1 - |C|), \\ \\
        1 + |A| + \dfrac{B^2}{4(1 + |C|)}, & B^2 < \min\{4(1 + |C|)^2,  -4AC(C^{-2} - 1)\}, \\ \\
        R(A, B, C), & \text{ Otherwise},
    \end{array}\right.\]
where 
\[R(A, B, C)= \displaystyle\left\{\begin{array}{ll}
        |A| + |B| - |C|, & |C|(|B| + 4|A|) \le |AB|, \\ \\
        -|A| + |B| + |C|, & |AB| \le |C|(|B| -4|A|), \\ \\
        (|C| + |A|) \sqrt{1 - \dfrac{B^2}{4AC}}, & \text{ Otherwise}.
    \end{array}\right. \] 
\end{lemma}

\begin{lemma}\label{lem:3}\cite{caratheodory1907,Pommerenke1975}
Let \( p \in \mathcal{P} \). Then, the following inequalities hold true
\begin{eqnarray}
|p_t|  &\leq & 2, \qquad \qquad \quad  t \geq 1, \nonumber\\
|p_{t+2k} - \rho p_t p_k^2| &\leq &  2(1 + 2\rho), \quad 0 \leq \rho \leq 1, \nonumber\\
\left| p_2 - \frac{p_1^2}{2} \right| &\leq& 2 - \frac{|p_1|^2}{2},\nonumber
\end{eqnarray}

and
\begin{equation}
|c_{n+k} - \mu c_n c_k| \leq 2 \max \{ 1, |2\mu - 1| \} =
\left\{
\begin{array}{ll}
2, & \text{if } 0 \leq \mu \leq 1, \\[6pt]
2|2\mu - 1|, & \text{Otherwise}.
\end{array}
\right.\nonumber
\end{equation}
\end{lemma}

\begin{lemma}\label{lem:5}\cite{zaprawa2021}
Let $w \in \mathcal{H}$, are said to be schwarz function such that $w(0)=0$ and $|w(z)|<1$ for all $z \in \mathbb{D}$, and have the following series:
\begin{equation}
w(z)=\sum_{n=1}^{\infty}b_nz^n \label{schwarzfunc}
\end{equation}
Then, the following inequalities hold true
\begin{eqnarray}
|b_1|  &\leq & 1,\nonumber\\
|b_2| & \leq & 1-|b_1|^2,\nonumber \\
|b_3| &\leq &  1-|b_1|^2-\frac{|b_2|^2}{1+|b_1|}. \nonumber
\end{eqnarray}
\end{lemma}
\section{Hankel determinant}
We begin by first establishing the bounds for the initial coefficients of the function $f \in \mathcal{S}^{\ast}_{B}$:

\begin{theorem}
Let $f \in \mathcal{S}^{\ast}_{B}$. Then, the following inequalities hold true 
\begin{equation}
|a_2| \leq 1, \quad |a_3| \leq \dfrac{3}{4}, \quad |a_4|\le\dfrac{19}{36}, \quad |a_5|\le\dfrac{101}{288}.\nonumber
\end{equation}
These inequalities are sharp.
\end{theorem}
\begin{proof}
Let $f \in \mathcal{S}^{\ast}_{B}$. Then there exists a Schwarz function $w(z)$ such that
\begin{equation}
    \dfrac{zf'(z)}{f(z)}= \dfrac{1}{1-\log(1+w(z))}. \label{eq:2.3}
\end{equation}
Suppose that $w(z)=(p(z)-1)/(p(z)+1)$, where $p \in \mathcal{P}$ is given by \eqref{pp}. Substituting this expression, along with \eqref{eq:1.1} and \eqref{pp}, into \eqref{eq:2.3}, we obtain the following relations between the coefficients:
\begin{eqnarray}
a_{2}&=&\dfrac{1}{2}p_{1},\label{a2}\\
a_{3}&=&\dfrac{1}{16}(p_1^2+4p_{2}),\label{a3}\\
a_{4}&=&\dfrac{1}{288}\left(p_{1}^{3}+12p_{1}p_{2}+48p_{3}\right),\label{a4}\\
a_5&=&-\dfrac{1}{4608}\left(7p_1^4-24p_1^2p_2-96p_1p_3-576p_4\right). \label{a5}
\end{eqnarray}
From \eqref{a2},\eqref{a3},\eqref{a4} and \eqref{a5}, it follows that
\begin{equation}
    |a_2|\le\dfrac{1}{2}|p_{1}|,\quad |a_3|\le\dfrac{1}{16}\left|p_1^2+4p_{2}\right|, \quad |a_4|\le\dfrac{1}{288}\left|p_{1}^{3}+12p_{1}p_{2}+48p_{3}\right|,\nonumber\end{equation}
    \begin{equation}|a_5|\le\dfrac{1}{4608}\left|7p_1^4-24p_1p_2-96p_1p_3-576p_4\right|.\nonumber\end{equation}
Using Lemma~\ref{lem:3}, we immediately deduce that  
\begin{equation}
|a_2| \leq 1, \quad |a_3| \leq \dfrac{3}{4}, \quad |a_4|\le\dfrac{19}{36}, \quad |a_5|\le\dfrac{101}{288}.\nonumber
\end{equation}
The sharpness of these inequalities is attained for the extremal function $f_1(z) \in \mathcal{S}^{\ast}_{B}$, defined in~\eqref{1.45}.
\end{proof}
We now derive the sharp bounds of Fekete-Szeg\"o functional $|a_{3}-\mu a_{2}^{2}|$ for $f \in \mathcal{S}^{\ast}_{B}$. Some recent results on the Fekete-Szeg\"o functional see \cite{Kowalczyk2014}.
\begin{theorem}
Let $f \in \mathcal{S}^{\ast}_{B}$. Then for any complex number $\mu \in \mathbb{C}$, the following inequality holds:
\begin{equation}
|a_{3} - \mu a_{2}^{2}| \leq \tfrac{1}{2} \max \left\{ 1, \left|\mu-\tfrac{3}{4} \right| \right\}.\nonumber
\end{equation}
This inequality is sharp.
\end{theorem}
\begin{proof}
Let $f \in \mathcal{S}^{\ast}_{B}$. Then from \eqref{a2}, \eqref{a3} and using lemma \ref{lem:3} we get,
\begin{equation}
    |a_3-\mu a_2^2|=\dfrac{1}{16}\left|(1-4\mu )p_1^2+4p_2\right|\leq\dfrac{1}{2} \max \left\{ 1, \left|\mu-\dfrac{3}{4} \right|\right\}.\nonumber
\end{equation}
The sharpness of these inequalities follows from the function $f_2(z) \in \mathcal{S}^{\ast}_{B}$, given by \eqref{1.466}. 
\end{proof}
Note that, when $\mu=1$, we have $\mathcal{H}_{2,1}(f)=a_{3}-a_{2}^{2}$. Hence, from the above theorem, the sharp bound for $\mathcal{H}_{2,1}(f)$ is given by $|\mathcal{H}_{2,1}(f)|\leq\tfrac{1}{2}$. Equality in this bound is attained for the function $f_2(z) \in \mathcal{S}^{\ast}_{B}$, as defined in \eqref{1.466}.

We now deduce the second order Hankel determinant associated with initial coefficients for $f \in \mathcal{S}^{\ast}_{B}$:
\begin{theorem}
Let $f \in \mathcal{S}^{\ast}_{B}$. Then
\begin{equation}
|\mathcal{H}_{2,2}(f)| \leq \dfrac{1}{4}. \label{eq:a2a4-a3^2}
\end{equation}
This inequality is sharp.
\end{theorem}
\begin{proof}
Let $f \in \mathcal{S}^{\ast}_{B}$. Using \eqref{a2},\eqref{a3} and \eqref{a4} in \eqref{h21 and h22} for $\mathcal{H}_{2,2}(f)$, we obtain   
\begin{equation}
\left| a_{2}a_{4} - a_3^2 \right| = \dfrac{1}{2304}\left|5 p_1^4+24 p_1^2 p_2 -192p_1 p_3 + 144 p_2^2 \right| \label{h22}
\end{equation}
Since the class $\mathcal{S}^{\ast}_{B}$ and $| a_{2}a_{4} - a_3^2 |$ are invariant under rotation. For instance, considering the function $f_\theta(z) := e^{-i \theta} f(e^{i \theta} z)$, where $f \in \mathcal{S}$ and $\theta \in \mathbb{R}$, we obtain
\begin{equation*}
    \left(a_2 a_4- a_3^2\right)_{\theta} = e^{4i\theta} \left(a_2 a_4- a_3^2\right)
\end{equation*}
Since $|(a_2 a_4- a_3^2)_{\theta}|=  |a_2 a_4- a_3^2|$, thus $|\mathcal{H}_{2,2}(f)|$ remains invariant under rotation. Now, by applying Lemma \ref{lem:1} in \eqref{h22}, we get  
\begin{equation}
\left| a_{2}a_{4} - a_3^2 \right| =\dfrac{1}{144}\left|12\zeta_2^2(\zeta_1^4+2\zeta_1^2-3)+12\zeta_1^2(\zeta_1^2-1)\zeta_2-5\zeta_1^4+48\zeta_1\zeta_3(1-\zeta_1^2)(1-|\zeta_2|^2)\right|, \label{a2a4-a3^2}
\end{equation}
Since $|\zeta_3|\le 1$, from \eqref{a2a4-a3^2}, we have the following cases for $\zeta_{1} = 0$ and $\zeta_{1} = 1$:
\[|a_2a_4-a_3^2|=  \left\{ \begin{array}{ll}
\dfrac{|\zeta_2|^2}{4}\leq \dfrac{1}{4},  &\zeta_{1}=0,\\\\
\dfrac{5}{144}, &\zeta_{1}=1.
\end{array}\right.
\]
For $\zeta_1 \in (0,1)$, applying the triangle inequality to \eqref{a2a4-a3^2} and using $|\zeta_3|\le 1$, we obtain
\begin{equation}
    |a_2a_4-a_3^2| \leq \dfrac{1}{3} \zeta_{1}(1-\zeta_{1}^{2})\; \Psi(A,B,C), \label{a2a4-a3^2 phi}
\end{equation}
where 
\begin{equation*} 
    \Psi(A,B,C):= |A+B\zeta_2+C\zeta_2 ^2| + 1-|\zeta_2|^2,
\end{equation*}
and
\[
    A=-\dfrac{5\zeta_1^3}{48(1-\zeta_1^2)},\quad B=\dfrac{1}{4}\zeta_1, \quad C=-\dfrac{3+\zeta_1^2}{4\zeta_1}. 
\]
We now consider the cases in Lemma \ref{lem:2}. For $\zeta_{1} \in (0,1)$, it can be observed that $AC \ge 0$ and $|B|\ge 2(1-|C|).$
Hence, by Lemma \ref{lem:2}, we have
\[\Psi (A,B,C) \le |A|+|B|+|C|.\] 
Substituting this estimate into \eqref{a2a4-a3^2 phi}, we get
\begin{eqnarray*}
    |a_2a_4-a_3^2| &\leq& \dfrac{1}{3} \zeta_{1}(1-\zeta_{1}^{2})\left(|A|+|B|+|C|\right)\\
    &\le&\dfrac{1}{3} \zeta_{1}(1-\zeta_{1}^{2})\left(\left|\dfrac{5\zeta_{1}^{3}}{48(1-\zeta_{1}^2)}\right|+\left|\dfrac{\zeta_1}{4}\right|+\left|\dfrac{3+\zeta_{1}^{2}}{4\zeta_{1}}\right|\right)\le\dfrac{1}{4}.
\end{eqnarray*}
\\
Thus, from the above result, it follows that the inequality \eqref{eq:a2a4-a3^2} holds true. The bound is sharp, and equality is achieved for the function $f_{2}(z) \in \mathcal{S}^{\ast}_{B}$, as defined in~\eqref{1.466}.
\end{proof}

We now proceed to establish the sharp bounds for the logarithmic coefficients of the function $f \in \mathcal{S}^{\ast}_{B}$:
\begin{theorem}
Let $f \in \mathcal{S}^{\ast}_{B}$. Then, the following inequalities for the logarithmic coefficients are true:
\[
|\gamma_1| \leq \dfrac{1}{2},\quad|\gamma_2| \leq \dfrac{1}{4}, \quad |\gamma_3| \leq \dfrac{1}{8}.
\]
These bounds are sharp.
\end{theorem}
\begin{proof}
    Let $f \in \mathcal{S}^{\ast}_{B}$. Substituting \eqref{a2},\eqref{a3} and \eqref{a4} into \eqref{log coeff}, we obtain
    \[\gamma_{1}=\dfrac{1}{4}p_{1},\quad \gamma_{2}=-\dfrac{1}{32}(p_1^2-4p_{2}), \quad \gamma_3=\dfrac{1}{144}\left(p_{1}^{3}-6p_{1}p_{2}+12p_{3}\right).\]
Applying Lemma~\ref{lem:3}, we obtain
\begin{equation}
    |\gamma_1|\le1/2, \quad |\gamma_2|\leq 1/4, \quad |\gamma_3| \leq 1/8.\nonumber
\end{equation}
These bounds are sharp, and equality is achieved for the functions $f_1(z)$ and $f_2(z) \in \mathcal{S}^{\ast}_{B}$, given by \eqref{1.45} and \eqref{1.466}, respectively. 
\end{proof}
The second-order Hankel determinants defined for the initial coefficients, in \eqref{h21 and h22}, through modification of the coefficients, the second Hankel determinant can be computed for different values of the coefficients. For instance, using equation \eqref{log coeff} in \eqref{1.10}, we can derive the Hankel determinant for logarithmic coefficients, given by
\begin{equation}
    \mathcal{H}_{2,1}(F_{f}/2) = \begin{vmatrix}
\gamma_1 & \gamma_{2} \\
\gamma_{2} & \gamma_{3}
\end{vmatrix} = \gamma_1 \gamma_3-\gamma_2^2 = \dfrac{1}{4}\left(a_2 a_4- a_3^2+\dfrac{1}{12}a_2^4\right).\label{3.11}
\end{equation}
It is important to note that $|\mathcal{H}_{2,1}(F_{f}/2)|$ remains invariant under rotation. For instance, considering the function $f_\theta(z) := e^{-i \theta} f(e^{i \theta} z)$, where $f \in \mathcal{S}$ and $\theta \in \mathbb{R}$, the following relation is derived:
\begin{equation}
    \mathcal{H}_{2,1} \left( F_{f_\theta} / 2 \right) = \dfrac{e^{4i\theta}}{4} \left(a_2 a_4- a_3^2+\dfrac{1}{12}a_2^4\right)=e^{4i\theta}\mathcal{H}_{2,1} \left( F_{f}/ 2 \right). \label{h21log}
\end{equation}
We now proceed to determine the sharp bounds of the second-order Hankel determinant corresponding to the logarithmic coefficients for $f \in \mathcal{S}^{\ast}_{B}$. 
\begin{theorem}
Let $f \in \mathcal{S}^{\ast}_{B}$. Then
\begin{equation}
    |\mathcal{H}_{2,1}(F_{f}/2)| \leq \dfrac{1}{16}. \label{eq:2.1}
\end{equation} 
This inequality is sharp.
\end{theorem}
\begin{proof} 
Let $f \in \mathcal{S}^{\ast}_{B}$. Since the class $\mathcal{S}^{\ast}_{B}$ is invariant under rotation and
$\mathcal{H}_{2,1}(F_f/2)$ is given by \eqref{h21log}, it follows that
$\lvert \mathcal{H}_{2,1}(F_f/2) \rvert$ is also rotationally invariant.
Therefore, without loss of generality, we may assume that $a_2 \ge 0$.
Consequently, by \eqref{a2}, we have $p_1 \ge 0$, which, in view of \eqref{l1},
implies that $\zeta_1 \in [0,1]$. Hence, upon substituting \eqref{a2},
\eqref{a3}, and \eqref{a4} into \eqref{3.11}, we obtain
\begin{eqnarray}
    \mathcal{H}_{2,1}(F_f/2)&=&\left(\dfrac{1}{2} a_{2}\right)\left(\dfrac{1}{2}\left(a_{4}-a_{2}a_{3}+\dfrac{1}{3}a_{2}^{3}\right)\right)-\left(\dfrac{1}{2}\left(a_{3}-\dfrac{1}{2}a_{2}^{2}\right)\right)^{2}\nonumber\\
    &=&\dfrac{1}{4}\left(a_2 a_4- a_3^2+\dfrac{1}{12}a_2^4\right)\nonumber\\
    &=&\dfrac{1}{9216}\left(7p_1^4-24p_1^2p_2-144p_2^2+192p_1p_3\right)\label{3.9}
\end{eqnarray}
Now, applying Lemma \ref{lem:1} to \eqref{3.9}, we get 
\begin{align}
\mathcal{H}_{2,1}(F_f/2)
&=\dfrac{1}{576}\Big(12\zeta_2^2(\zeta_1^4+2\zeta_1^2-3)
+ 12\zeta_1^2\zeta_2(1 - \zeta_1^2)+ \notag\\
& \quad \quad 7\zeta_1^4+48\zeta_1(1-\zeta_1^2)(1-|\zeta_2|^2)\zeta_3
\Big) \label{eq:3.12}
\end{align}
\begin{enumerate}
\item Since $|\zeta_3|\le 1$, from \eqref{eq:3.12}, we have the following cases for $\zeta_{1} = 0$ and $\zeta_{1} = 1$:
\[|\mathcal{H}_{2,1}(F_f/2)|=  \left\{ \begin{array}{ll}
\dfrac{|\zeta_2|^2}{16}\leq \dfrac{1}{16},  &\zeta_{1}=0,\\ \\
\dfrac{7}{576}, &\zeta_{1}=1.
\end{array} \right.
\]
  \item When $\zeta_{1} \in (0,1)$, since $|\zeta_{3}| \leq 1$, applying the triangle inequality to \eqref{eq:3.12} gives
\begin{align}
    |\mathcal{H}_{2,1}(F_f/2)|
    &\leq \dfrac{1}{576}\big|12\zeta_2^2(\zeta_1^4+2\zeta_1^2-3)
    +12\zeta_1^2\zeta_2(1-\zeta_1^2) +\notag \\
   &\quad \quad 7\zeta_1^4+48\zeta_1(1-\zeta_1^2)(1-|\zeta_2|^2)\zeta_3\big| \notag\\
    &= \dfrac{1}{7} \zeta_{1}(1-\zeta_{1}^{2})\Psi(A,B,C) \label{eq:2.5}
\end{align}
where 
\begin{eqnarray*}
        \Psi(A,B,C)&:=& |A+B\zeta_2+C\zeta_2 ^2| + 1-|\zeta_2| ^2,
\end{eqnarray*}
and
\[
    A=\dfrac{7\zeta_{1}^{3}}{48(1-\zeta_{1}^2)},\quad
    B=\dfrac{\zeta_1}{4},\quad
    C=-\dfrac{3+\zeta_{1}^{2}}{4\zeta_{1}}.
\]
Since $AC < 0$, by applying \textit{Case 2} of Lemma~\ref{lem:2}, we proceed as follows.  
We define
\[
T_2(\zeta_1) := -4AC\left(\frac{1}{C^2} - 1\right) - B^2 = -\frac{\zeta_1^2(18-\zeta_1^2)}{12(3 + \zeta_1^2)} \leq 0,
\]
which gives
\[
-4AC\left(\frac{1}{C^2} - 1\right) \leq B^2.
\]

\begin{enumerate}[label=\Alph*.]
    \item For each $\zeta_1 \in (0,1)$
    \[T_1(\zeta_1):=|B|-2(1-|C|)=\dfrac{3}{2\zeta}+\dfrac{3\zeta_1}{4}-2>0,\]
    implying $|B|>2(1-|C|).$
    Furthermore,
    \[T_2(\zeta_1):=-4AC\left(\frac{1}{C^2}-1\right)-B^2=-\dfrac{\zeta_1^2(18-\zeta_1^2)}{12(3+\zeta_1^2)}\leq 0,\]
    which gives $-4AC\left(\frac{1}{C^2}-1\right)\leq B^2.$
    Thus, $T_1(\zeta_1) \cap T_2(\zeta_1)=\emptyset$, and this case does not occur for any $\zeta_{1} \in (0,1)$, as stated in Lemma~\ref{lem:2}.
    \item  For $\zeta_{1} \in (0,1)$, we have
    \[T_3(\zeta_1):=4(1+|C|)^2=\dfrac{(3+4\zeta_1+\zeta_1^2)^2}{4\zeta_1^2}>0,\]
     \[T_4(\zeta_1):=-4AC\left(\frac{1}{C^2}-1\right)=-\dfrac{7\zeta_1^2(9-\zeta_1^2)}{48(3+\zeta_1^2)}<0.\]
     Therefore, $\min\{T_3(\zeta_1),T_4(\zeta_1)\}= T_4(\zeta_1)$. Since $-4AC\left(\frac{1}{C^{2}} - 1\right) \leq B^{2}$, this case is also not valid for any $\zeta_{1} \in (0,1)$.
    \item Considering 
    \[T_5(\zeta_1):=|A B|-|C|(|B|+4|A|)=-\dfrac{12+20\zeta_1^2+3\zeta_1^4}{64(1-\zeta_1^2)}<0,\]
    we get $|AB|<|C|(|B|+4|A|)$, implying this case is impossible for $\zeta_{1} \in (0,1)$.
     \item Finally, define
     \[
     T_6(\zeta_1):= |A B|-|C|(|B|-4|A|)=-\dfrac{36-108\zeta_1^2-47\zeta_1^4)}{192(1-\zeta_1^2)} \leq 0.\nonumber
     \]
    which holds for $0 < \zeta_1 \le \zeta' = \left(\sqrt{\dfrac{6}{47}(8\sqrt{2}-9})\right)$.
    \end{enumerate}
Hence, by Lemma \ref{lem:2}, 
\[\Psi (A,B,C) \le |A|+|B|+|C|.\] 
Using this in \eqref{eq:2.5}, we get 
\begin{eqnarray*}
    |\mathcal{H}_{2,1}(F_f/2)| &\leq& \dfrac{1}{7} \zeta_{1}(1-\zeta_{1}^{2})\big(|A|+|B|+|C|\big)\\
    &=&\dfrac{1}{576}(36-12\zeta_1^2-31\zeta_1^4)\\
    &\le& \dfrac{1}{16}\approx0.0625.
\end{eqnarray*}

For $\zeta' < \zeta_{1} < 1$, applying Lemma~\ref{lem:2} again yields
    \begin{eqnarray*}
        |\mathcal{H}_{2,1}(F_f/2)|&\leq& \dfrac{1}{12} \zeta_1(1-\zeta_1^2) \left(|C|+|A|\right)\sqrt{1-\dfrac{B^2}{4AC}}\\
          &=&\dfrac{1}{168} \sqrt{\dfrac{6+\zeta_1^2}{21 + 7\zeta_1^2}} \left( 5 \zeta_1^4+24\zeta_1^2 -36\right) := \phi_2(\zeta_1)
    \end{eqnarray*}
    For $\zeta_{1} \in (\zeta', 1)$, we find that
        \[
        \phi_2(\zeta_1) \le0.0516512 \le \frac{1}{16} \quad \text{at  }\; \zeta_1=\sqrt{\frac{6}{47}(8\sqrt{2}-9)}.\]
        \end{enumerate}
Therefore, it follows that the inequality \eqref{eq:2.1} holds. In lemma \ref{lem:1}, on replacing $p_1=p_3=0$ and $p_2=2$. The corresponding extremal function $f\in\mathcal{S}^{\ast}_{B}$ described as 
\begin{equation}
    \dfrac{zf'(z)}{f(z)}=\dfrac{1}{1-\log(2p(z)/(p(z)+1))} \label{zf'(z)/f(z)}
\end{equation}
Where $p(z)$ is given in \eqref{p(z)} with $\zeta_{1}=0$ and $\zeta_{2}=1$, we get $p(z)=(1+z^2)/(1-z^2)$.  On solving \eqref{zf'(z)/f(z)}, we get the function~\eqref{1.466}.

\end{proof}
We now proceed to establish the bounds for the Second-order Hankel determinants, where the entries are the logarithmic coefficients of the inverse of $f\in\mathcal{S}^{\ast}_{B}$. 
Using the equation \eqref{loginvcoef123} in \eqref{1.13}, the logarithmic coefficients of the inverse functions are derived as follows:
\begin{equation}
    \mathcal{H}_{2,1}(F_{f^{-1}}/2) = \begin{vmatrix} 
    \Gamma_1 & \Gamma_2 \\
    \Gamma_2 & \Gamma_3 
    \end{vmatrix} = \Gamma_1 \Gamma_3 - \Gamma_2^2 = \dfrac{1}{48} \left( 13 a_2^4 - 12 a_2^2 a_3 - 12 a_3^2 + 12 a_2 a_4 \right). \label{inv coeff}
\end{equation}

Similarly, we can verify that $\lvert \mathcal{H}_{2,1}(F_{f^{-1}}/2) \rvert$ is also invariant under rotation. Indeed, for the rotated function $f_\theta(z) := e^{-i\theta} f(e^{i\theta} z)$, where $f \in \mathcal{S}$ and
$\theta \in \mathbb{R}$, we obtain
\begin{equation}
\mathcal{H}_{2,1}(F_{f_\theta^{-1}}/2)
= \dfrac{e^{4i\theta}}{48}
\left( 13 a_2^4 - 12 a_2^2 a_3 - 12 a_3^2 + 12 a_2 a_4 \right)
= e^{4i\theta}\mathcal{H}_{2,1}(F_{f^{-1}}/2).
\label{h21inverselog}
\end{equation}

In the following, we obtain the sharp bounds of the second-order Hankel determinant related to the logarithmic coefficients of the inverse function for functions $f$ belonging to the class $\mathcal{S}^{\ast}_{B}$:
\begin{theorem}
Let $f \in \mathcal{S}^{\ast}_{B}$. Then
    \begin{equation}
        |\mathcal{H}_{2,1}(F_{f^{-1}}/2)| \leq \dfrac{43}{576}. \label{eq:3.1}
    \end{equation} 
This inequality is sharp.   
\end{theorem}
\begin{proof}
Let $f \in \mathcal{S}^{\ast}_{B}$. In view of the rotational invariance of the class $\mathcal{S}^{\ast}_{B}$ and from \eqref{h21inverselog}, $\lvert \mathcal{H}_{2,1}(F_{f^{-1}}/2) \rvert$ is rotationally invariant. Accordingly, without loss of generality, we assume that $a_2 \ge 0$. It then follows from \eqref{a2} that $p_1 \ge 0$, and hence, by \eqref{l1}, $\zeta_1 \in [0,1]$. Substituting \eqref{a2}, \eqref{a3}, and \eqref{a4} into
\eqref{inv coeff}, we obtain

\begin{eqnarray} 
    \mathcal{H}_{2,1}(F_{f^{-1}}/2)&=& \Gamma_{1}\Gamma_{3}-\Gamma_{2}^{2} \nonumber\\
    &=&\dfrac{1}{48}\left(13a_2^4-12a_2^2a_3-12a_3^2+12a_2a_4\right) \nonumber\\
    &=&\dfrac{1}{9216}\left(115p_1^4-168p_1^2p_2-144p_2^2+192p_1p_3\right) \label{111}
\end{eqnarray}
Applying Lemma~\ref{lem:1} in \eqref{111}, we obtain
\begin{equation}
    \mathcal{H}_{2,1}(F_{f^{-1}}/2)=\dfrac{1}{576}(12\zeta_2^2(\zeta_1^4+2\zeta_1^2-3)+60\zeta_1^2(\zeta_1^2-1)\zeta_2+43\zeta_1^4+48\zeta_1\zeta_3(1-\zeta_1^2)(1-|\zeta_2|^2)) \label{eq:3.13}
\end{equation}
\begin{enumerate}
\item Since $|\zeta_3|\le 1$, from \eqref{eq:3.13}, we have the following inequality for $\zeta_{1} = 0$ and $\zeta_{1} = 1$:
\[|\mathcal{H}_{2,1}(F_{f^{-1}}/2)|=  \left\{ \begin{array}{ll}
\dfrac{|\zeta_2|^2}{16}\leq \dfrac{1}{16},  &\zeta_{1}=0,\\ \\
\dfrac{43}{576}, &\zeta_{1}=1.
\end{array}\right.
\]
\item For $\zeta_1 \in (0,1)$ and $|\zeta_3|\le 1$, applying the triangle inequality to \eqref{eq:3.13} yields
\begin{equation}
    |\mathcal{H}_{2,1}(F_{f^{-1}}/2)| \leq \dfrac{1}{12} \zeta_{1}(1-\zeta_{1}^{2})\; \Psi(A,B,C), \label{eq:3.5}
\end{equation}
where 
\begin{equation*} 
    \Psi(A,B,C):= |A+B\zeta_2+C\zeta_2 ^2| + 1-|\zeta_2|^2,
\end{equation*}
and
\begin{equation}
    A=\dfrac{43\zeta_1^3}{48(1-\zeta_1^2)},\; B=-\dfrac{5}{4}\zeta_1, \;C=-\dfrac{3+\zeta_1^2}{4\zeta_1}. \label{eq:3.77}
\end{equation}
Since $AC < 0$, we analyze the following subcases based on Lemma~\ref{lem:2}:
\begin{enumerate}[label=\Alph*.]
    \item For each $\zeta_1 \in (0,1)$,
    \[T_1(\zeta_1):=|B|-2(1-|C|)=\dfrac{3}{2\zeta_1}+\dfrac{7\zeta_1}{4}-2>0,\]
    implying that $|B|>2(1-|C|)$.
    Furthermore,
    \[T_2(\zeta_1):=-4AC\left(\frac{1}{C^2}-1\right)-B^2=-\dfrac{\zeta_1^2(153+8\zeta_1^2)}{12(3+\zeta_1^2)}\leq 0,\]
    implying
    \[-4AC\left(\frac{1}{C^2}-1\right)\leq B^2.\]
    Since $T_1(\zeta_1) \cap T_2(\zeta_1)=\emptyset$. this subcase does not occur for any $\zeta_1 \in (0,1)$.
    
    \item For $\zeta_1 \in (0,1)$, 
    \[T_3(\zeta_1):=4(1+|C|)^2=\dfrac{(3+4\zeta_1+\zeta_1^2)^2}{4\zeta_1^2}>0,\]
     \[T_4(\zeta_1):=-4AC\left(\frac{1}{C^2}-1\right)=-\dfrac{43\zeta_1^2(9-\zeta_1^2)}{48(3+\zeta_1^2)}<0.\]
     Thus, $\min\{T_3(\zeta_1),T_4(\zeta_1)\}= T_4(\zeta_1)$.
    From above subcase, we already know that $-4AC\left(\frac{1}{C^2}-1\right)\leq B^2$, so this subcase also does not occur.
    
    \item For $\zeta_1 \in (0,1)$, 
    \[T_5(\zeta_1):=|A B|-|C|(|B|+4|A|)=-\dfrac{180+396\zeta_1^2-103\zeta_1^4}{192(1-\zeta_1^2)}<0.\]
    implying $|AB|<|C|(|B|+4|A|)$, hence, this subcase is also not possible.
    
     \item For $\zeta_1 \in (0,1)$, we take 
     \[T_6(\zeta_1):= |A B|-|C|(|B|-4|A|)=-\dfrac{60-212\zeta_1^2-149\zeta_1^4}{64(1-\zeta_1^2)} \leq 0.\]
    which holds for
    \[0 < \zeta_1 \le \zeta' = \sqrt{\frac{2}{149}(2\sqrt{1261}-53})\approx 0.491827.\]
    \end{enumerate}
    \end{enumerate}
Therefore, by Lemma \ref{lem:2},
\[\Psi (A,B,C) \le -|A|+|B|+|C|.\] 
Substituting this into \eqref{eq:3.13}, we obtain
\begin{eqnarray*}
    |\mathcal{H}_{2,1}(F_{f^{-1}}/2)| &\leq& \dfrac{1}{12} \zeta_{1}(1-\zeta_{1}^{2})\left(-|A|+|B|+|C|\right)\\
    &=&\dfrac{1}{576}(36+36\zeta_1^2-115\zeta_1^4)\le \dfrac{31}{460}\approx0.0673913.
\end{eqnarray*}
Applying Lemma \ref{lem:2} for $\zeta'<\zeta_1<1$, we get
    \begin{eqnarray*}
         |\mathcal{H}_{2,1}(F_{f^{-1}}/2)|&\leq& \dfrac{1}{12} \zeta_1(1-\zeta_1^2) \big(|C|+|A|\big)\sqrt{1-\dfrac{B^2}{4AC}}\\
          &=&\dfrac{1}{288\sqrt{43}} \sqrt{\dfrac{51-8\zeta_1^2}{3 + \zeta_1^2}} \left( 36 -24\zeta_1^2 +31\zeta_1^4 \right) = :\phi_2(\zeta_1)
    \end{eqnarray*}
        For $\zeta_1 \in (\zeta', 1)$, we find $\phi_2(\zeta_1) \le \dfrac{43}{576}\approx0.0746528$.\\
Hence, inequality \eqref{eq:3.1} holds. In lemma \ref{lem:1}, on replacing $p_1=p_2=p_3=0$. The corresponding extremal function $f\in\mathcal{S}^{\ast}_{B}$ described as 
\begin{equation}
    \dfrac{zf'(z)}{f(z)}=\dfrac{1}{1-\log(2p(z)/(p(z)+1))} \label{zf'(z)/f(z)forinverse}
\end{equation}
Where $p(z)$ is given in \eqref{firstp(z)} with $\zeta_{1}=1$, we get $p(z)=(1+z)/(1-z)$.  On solving \eqref{zf'(z)/f(z)forinverse}, we get the function~\eqref{1.45}.
\end{proof}

\section{\hspace{5pt} Toeplitz Determinant}
We now establish the bounds for second order Toeplitz determinant associated with initial coefficients for $f \in \mathcal{S}^{\ast}_{B}$:
\begin{theorem}
Let $f \in \mathcal{S}^{\ast}_{B}$. Then
\begin{equation}
   |\mathcal{T}_{2,1}(f)| \leq 2.\nonumber
\end{equation} 
This inequality is sharp.
\end{theorem}
\begin{proof}
    Let $f \in \mathcal{S}^{\ast}_{B}$. Substituting \eqref{a2} into \eqref{t22} for $\mathcal{T}_{2,1}(f)$, we obtain
    \[
        |1-a_2^2|\le1+|a_2|^2=1+\dfrac{|p_1|^2}{4}
    \]
    Applying Lemma \eqref{lem:3}, which gives $|1-a_2^2|\leq2$. This bound is sharp, and equality is achieved for the function $f_3(z) \in \mathcal{S}^{\ast}_{B}$, defined in \eqref{1.477}.
\end{proof}
\begin{theorem}
Let $f \in \mathcal{S}^{\ast}_{B}$. Then
\begin{equation}
   |\mathcal{T}_{2,2}(f)| \leq \dfrac{25}{16}.\nonumber
\end{equation} 
This inequality is sharp.
\end{theorem}
\begin{proof}
Let $f \in \mathcal{S}^{\ast}_{B}$. Substituting \eqref{a2} and \eqref{a3} into \eqref{t22} for $\mathcal{T}_{2,2}(f)$, we have
\begin{eqnarray}
    \left|a_{3}^2-a_{2}^{2}\right|&\le&|a_3|^2+|a_2|^2\nonumber \\
    &\le&\dfrac{1}{256}\left|p_1^4+16p_2^2+8p_1^2p_2\right|+\dfrac{1}{4}|p_1|^2\nonumber
\end{eqnarray}
Applying Lemma \ref{lem:3}, we obtain
\[
|a_3^2-a_2^2|\leq\dfrac{25}{16}.
\]
This bound is sharp, and equality is achieved for the function $f_3(z)\in\mathcal{S}^{\ast}_{B}$, defined in \eqref{1.477}.
\end{proof}
\begin{theorem}
    Let $f \in \mathcal{S}^{\ast}_{B}$. Then
\begin{equation}
   |\mathcal{T}_{2,3}(f)| \leq \dfrac{545}{648}.\nonumber
\end{equation} 
This inequality is sharp.
\end{theorem}
\begin{proof}
Let $f\in \mathcal{S}^{\ast}_{B}$. Then, there exists Schwarz function $w(z)$ satisfying \eqref{schwarzfunc}. By solving and comparing the coefficients of $f(z)$ and $w(z)$  from \eqref{eq:2.3}, we obtain
\begin{equation}
    a_2=b_1,\quad a_3=\dfrac{3b_1^2}{4}+\dfrac{b_2}{2},\quad a_4= \dfrac{19b_1^3}{36}+\dfrac{5b_1b_2}{6}+\dfrac{b_3}{3}.\label{a2a3a4 in terms of w}
\end{equation}
    Substituting the expressions from \eqref{a2a3a4 in terms of w} into \eqref{t22} for $\mathcal{T}_{2,3}(f)$ and simplifying, we get
    \begin{eqnarray*}
    T_{2,3}(f)&=&a_3^2-a_4^2\\
    &=&-\dfrac{1}{1296}\big(361b_1^6-729b_1^4+1140b_1^4b_2-972b_1^2b_2-342b_2^2\\
    & &  \qquad+900b_1^2b_2^2+456b_1^3b_3+720b_1b_2b_3+144b_3^2\big)
    \end{eqnarray*}
    Applying triangle inequality and Lemma \ref{lem:5}, we obtain
    \begin{eqnarray*}
        |T_{2,3}(f)|&\leq&\dfrac{1}{1296}\big[361|b_1|^6+729|b_1|^4+1140|b_1|^4|b_2|+972|b_1|^2|b_2|+342|b_2|^2\\
    & &  \qquad+900|b_1|^2|b_2|^2+456|b_1|^3|b_3|+720|b_1||b_2||b_3|+144|b_3|^2\Big]\\
    &\leq&\dfrac{1}{1296}\Big[361|b_1|^6+729|b_1|^4+1140|b_1|^4(1-|b_1|^2)+972|b_1|^2(1-|b_1|^2)\\
    & &  \qquad+342(1-|b_1|^2)^2+900|b_1|^2(1-|b_1|^2)^2+456|b_1|^3\left(1-|b_1|^2-\dfrac{|b_2|^2}{1+|b_1|}\right)\\
    & & \qquad+720|b_1|(1-|b_1|^2)\left(1-|b_1|^2-\dfrac{|b_2|^2}{1+|b_1|}\right)+144\left(1-|b_1|^2-\dfrac{|b_2|^2}{1+|b_1|}\right)^2\Big]
    \end{eqnarray*}
    Setting $|b_1|:=x$ and $|b_2|:=y$ then we get
    \begin{eqnarray}
        |T_{2,3}(f)|&\leq&\dfrac{1}{1296}\Big[361x^6+729x^4+1140x^4(1-x^2)+972x^2(1-x^2)+342(1-x^2)^2\nonumber\\ 
    & &  \qquad+900x^2y^2+456x^3\left(1-x^2-\tfrac{y^2}{1+x}\right)\nonumber\\
    & & \qquad+720xy\left(1-x^2-\tfrac{y^2}{1+x}\right)+144\left(1-x^2-\tfrac{y^2}{1+x}\right)^2\Big]\nonumber\\
    &\leq&\dfrac{1}{1296}M(x,y), \label{Mxy}
    \end{eqnarray}
    where 
    \begin{eqnarray*}
        M(x,y)&=&\Big[361x^6+729x^4+1140x^4(1-x^2)+972x^2(1-x^2)+342(1-x^2)^2\\
    & &  \qquad+900x^2y^2+456x^3\left(1-x^2-\tfrac{y^2}{1+x}\right)\\
    & & \qquad+720xy\left(1-x^2-\tfrac{y^2}{1+x}\right)+144\left(1-x^2-\tfrac{y^2}{1+x}\right)^2\Big]
    \end{eqnarray*}
    In view of Lemma \ref{lem:5}, we have $\phi=\{(x,y):0\leq y\leq 1-x^2,0\leq x\leq1\}$, our aim is to establish the maximum value of $M(x,y)$ in the region $\phi$. Therefore, the critical point of $M(x,y)$ satisfies the conditions
    \[
        \dfrac{\partial M}{\partial x}=0 \quad \text{and} \quad\dfrac{\partial M}{\partial y}=0
    \]
    It can be observed that there are no solutions of $M(x,y)$ inside the interior of $\phi$, hence maximum of $M(x,y)$ must occur aon the boundary of $\phi$.
    \[
        M(x,0)=361x^6+729x^4-456x^5+456x^3+144(1-x^2)^2\leq1090; \quad 0\leq x\leq1,
    \]
    \[
    M(0,y)=144y^4+54y^2+144\leq342; \quad 0\leq y\leq 1.
    \]
    and \[
    M(x,1-x^2)\leq1090; \quad 0\leq x\leq1.
    \]
    Hence, we get $M(x,y)\leq1090$, by substituting maximum value of $M(x,y)$ in equation \eqref{Mxy}, we obtain
    \[
    |T_{2,3}(f)|\leq\dfrac{1090}{1296}=\dfrac{545}{648}\approx0.8410493827.
    \]
    This bound is sharp, and equality is achieved for the function $f_3(z)\in\mathcal{S}^{\ast}_{B}$, defined in \eqref{1.477}.
 \end{proof}

We can evaluate the Toeplitz determinant for various coefficient sets, by altering the coefficients. For example, by substituting equation \eqref{log coeff} into \eqref{1.12}, we can obtain the determinant corresponding to the logarithmic coefficients.
\begin{equation} \label{4.1}
    \mathcal{T}_{2,1}(F_{f}/2) = \begin{vmatrix}
\gamma_1 & \gamma_{2} \\
\gamma_{2} & \gamma_{1}
\end{vmatrix} = \gamma_1^2 - \gamma_2^2 = \frac{1}{16}\left(4a_2^2 - a_2^4 - 4a_3^2 + 4a_2^2a_3\right).
\end{equation}
We now proceed to establish the sharp bounds of the second-order Toeplitz determinant associated with the logarithmic coefficients for $f \in \mathcal{S}^{\ast}_{B}$:
\begin{theorem}
Let $f \in \mathcal{S}^{\ast}_{B}$. Then
\begin{equation}
   |\mathcal{T}_{2,1}(F_{f}/2)| \leq \frac{17}{64}.\nonumber
\end{equation} 
This inequality is sharp.
\end{theorem}

\begin{proof}
Let $f\in \mathcal{S}^{\ast}_{B}$. Substituting the expressions from \eqref{a2a3a4 in terms of w} into \eqref{log coeff} and simplifying, we get
\begin{eqnarray}
    (\gamma_1^2-\gamma_2^2)&=&-\dfrac{1}{16}\left(a_2^4-4a_2^2+4a_3^2-4a_2^2a_3\right)\nonumber\\
    &=&-\dfrac{1}{64}\left(b_1^4+4b_1^2b_2-16b_1^2+4b_2^2\right) \label{98}
\end{eqnarray}
Applying Lemma \ref{lem:5} to \eqref{98}, we obtain
\begin{eqnarray*}
|\gamma_1^2-\gamma_2^2|&\leq&\dfrac{1}{64}(|b_1|^4+4|b_1|^2|b_2|+16|b_1|^2+4|b_2|^2)\\
&\leq&\dfrac{1}{64}\left(|b_1|^4+4|b_1|^2(1-|b_1|^2)+16|b_1|^2+4(1-|b_1|^2)^2\right)\\
    &\leq&\dfrac{1}{64}\left(|b_1|^4+12|b_1|^2+4\right)
\end{eqnarray*}
Setting $\zeta:=|b_1|$, we have
\[|\gamma_1^2-\gamma_2^2|\leq\dfrac{1}{64}\left(\zeta^4+12\zeta^2+4\right)\]
Since $\zeta \in [0,1]$,
\[|\mathcal{T}_{2,1}(F_{f}/2)| \leq \frac{17}{64}\]
This bound is sharp, and equality is achieved for the function $f_3(z)\in\mathcal{S}^{\ast}_{B}$, defined in \eqref{1.477}. 
\end{proof}
We now evaluate the Toeplitz determinant corresponding to the logarithmic coefficients of inverse functions, by applying equation \eqref{loginvcoef123} in \eqref{1.14}.
\begin{equation} \label{4.2}
    \mathcal{T}_{2,1}(F_{f^{-1}}/2) = \begin{vmatrix}
\Gamma_1 & \Gamma_{2} \\
\Gamma_{2} & \Gamma_{1}
\end{vmatrix} = \Gamma_1^2 - \Gamma_2^2 = -\frac{1}{16}\left(9a_2^4 - 4a_2^2 + 4a_3^2 - 12a_2^2a_3\right).
\end{equation}
Next, we establish the sharp bounds for the second-order Toeplitz determinant corresponding to the logarithmic coefficients of inverse  $f \in \mathcal{S}^{\ast}_{B}$:
\begin{theorem}
Let $f \in \mathcal{S}^{\ast}_{B}$, then
\begin{equation}
    |\mathcal{T}_{2,1}(F_{f^{-1}}/2)| \leq \frac{25}{64}. \nonumber
\end{equation} 
This inequality is sharp.
\end{theorem}

\begin{proof} Let $f \in \mathcal{S}^{\ast}_{B}$. By substituting the expressions from \eqref{a2a3a4 in terms of w} into \eqref{4.2}, we obtain
\begin{eqnarray} 
    \mathcal{T}_{2,1}(F_{f^{-1}}/2)&=&\Gamma_1^2-\Gamma_2^2\nonumber\\
    &=& -\frac{1}{16}\left(9a_2^4-4a_2^2+4a_3^2-12a_2^2a_3\right)\nonumber\\
    &=&-\frac{1}{64}\left(9b_1^4+4b_2^2-16b_1^2-12b_1^2b_2\right)\label{4.4}
\end{eqnarray}
Applying Lemma~\ref{lem:5} to \eqref{4.4}, we obtain
\begin{eqnarray}
    |\mathcal{T}_{2,1}(F_{f^{-1}}/2)|&\le& \dfrac{1}{64} \left(9 |b_1^4| +4|b_2|^2+16|b_1|^2+12|b_1|^2|b_2|\right)\nonumber\\
    &\leq&\dfrac{1}{64}\left(9|b_1|^4+4(1-|b_1|^2)^2+16|b_1|^2+12|b_1|^2(1-|b_1|^2)\right)\nonumber\\
    &\le&\dfrac{1}{64}\left(|b_1|^4+20|b_1|^2+4\right). \nonumber
    \end{eqnarray}

Setting $\zeta:=|b_1|$, we have
\[
    |\mathcal{T}_{2,1}(F_{f^{-1}}/2)| \le \frac{1}{64}\left(\zeta^4+20\zeta^2+4\right),\nonumber
\]
Since $\zeta \in [0,1]$,
\[
    |\mathcal{T}_{2,1}(F_{f^{-1}}/2)| \le \frac{25}{64}.
\]
This bound is sharp, and equality is achieved for the function $f_3(z)\in\mathcal{S}^{\ast}_{B}$, defined in \eqref{1.477}.
\end{proof}

\section{Conclusion}
The determination of coefficient bounds for analytic and univalent functions remains a central topic in geometric function theory, with significant implications in image reconstruction, signal analysis, complex dynamics, and chaos modeling. Motivated by these applications, the present study introduces a new class of starlike functions associated with a balloon-shaped domain. For this class, sharp initial coefficient estimates, Fekete-Szegö type inequalities, and exact bounds for second-order Hankel and Toeplitz determinants-including logarithmic and inverse-logarithmic variants-are derived. These results highlight the deep connection between geometric structure and analytic behavior, offering a unified approach for extremal problems and potential extensions to higher-order determinants and related subclasses of univalent functions.

\end{document}